\documentclass[12pt]{amsart}

\usepackage{amsmath}
\usepackage{amsfonts}
\usepackage{amssymb}
\usepackage{amsthm}
\usepackage{hyperref}
\usepackage{enumerate}
\usepackage{cleveref}
\usepackage{mathrsfs}

\newtheorem{theorem}{Theorem}[section]
\newtheorem{lemma}[theorem]{Lemma}
\newtheorem{proposition}[theorem]{Proposition}

\newtheorem{definition}[theorem]{Definition}
\newtheorem{remark}[theorem]{Remark}
\newtheorem*{remark*}{Remark}

\Crefname{conjecture}{Conjecture}{Conjectures}

\theoremstyle{plain}

\theoremstyle{plain}

\newcommand{\N}{\mathbb{N}}

\newcommand{\Z}{\mathbb{Z}}
\newcommand\smod[1]{\ \left(\operatorname{mod} #1\right)}
\newcommand{\Q}{\mathbb{Q}}
\newcommand{\R}{\mathbb{R}}
\newcommand{\C}{\mathbb{C}}

\newcommand{{\D}}{\delta}
\newcommand{\F}{\mathcal{F}}

\newcommand{\bbH}{\mathbb{H}}

\newcommand{\eps}{\varepsilon}
\newcommand{\GL}{\operatorname{GL}}

\newcommand{\SL}{\operatorname{SL}}

\renewcommand{\Re}{\operatorname{Re}}

\newcommand{\SLZ}{\SL_2(\Z)}
\newcommand{\abcd}{\left(\begin{smallmatrix} a & b \\ c & d \end{smallmatrix}\right)}

\newcommand{\calQ}{\mathcal{Q}}

\newcommand{\z}{\tau}

\newcommand{\Gal}{\operatorname{Gal}}

\newcommand{\calM}{\mathcal{M}}
\renewcommand{\P}{\mathcal{P}}

\numberwithin{equation}{section}
\numberwithin{table}{section}

\author{Michael H. Mertens and Larry Rolen}
\address{Mathematisches Institut der Universit\"at zu K\"oln, Weyertal 86-90,
D-50931 K\"oln, Germany} 
\email{mmertens@math.uni-koeln.de}
\email{lrolen@math.uni-koeln.de}
\thanks{
The research of the first author leading to these results has received funding from the European Research Council under the European Union's Seventh Framework Programme (FP/2007-2013) / ERC Grant agreement n. 335220 - AQSER. The second author thanks the University of Cologne and the DFG for their generous support via the University of Cologne postdoc grant DFG Grant D-72133-G-403-151001011, funded under the Institutional Strategy of the University of Cologne within the German Excellence Initiative. }

\title[Class invariants for non-holomorphic modular functions]{On class invariants for non-holomorphic modular functions and a question of Bruinier and Ono}

\begin{document}

\begin{abstract}
Recently, Bruinier and Ono found an algebraic formula for the partition function in terms of traces of singular moduli of a certain non-holomorphic modular function. In this paper we prove that the rational polynomial having these singular moduli as zeros is (essentially) irreducible, settling a question of Bruinier and Ono. The proof uses careful analytic estimates together with some related work of Dewar and Murty, as well as extensive numerical calculations of Sutherland. 
\end{abstract}

\maketitle

\section{Introduction}
A \emph{partition} of a natural number $n$ is a non-increasing sequence of natural numbers which sum up to $n$. Let $p(n)$ denote the number of partitions of $n$. Despite its elementary definition, it is computationally infeasible to compute this number directly for large $n$. A much more efficient method for computing the partition number was offered by Euler's recursive formula for $p(n)$, obtained as a consequence of his pentagonal number theorem \cite{Euler}. Much later, Hardy and Ramanujan \cite{HR} found the asymptotic formula
\[p(n)\sim  \frac{1}{4n\sqrt{3}}e^{\pi\sqrt{\frac{2n}{3}}}\quad\text{as }n\rightarrow\infty,\]
developing and employing a device which has become essential in analytic number theory, the so-called Circle Method. By refining their method, Rademacher \cite{Rademacher} found his famous exact formula for $p(n)$,
\[p(n)=\frac{2\pi}{(24n-1)^\frac 34} \sum_{k=1}^\infty \frac{A_k(n)}{k}I_\frac 32\left(\frac{\pi\sqrt{24n-1}}{6k}\right),\]
where $I_\frac 32$ is the modified Bessel function of the first kind and $A_k(n)$ is a certain Kloosterman sum.

More recently, Bruinier and Ono \cite{BO} proved that $p(n)$ can also be computed as a finite sum of distinguished algebraic numbers. More precisely, they showed that
\[p(n)=\frac{1}{24n-1}\sum\limits_{Q\in\calQ_{{\D}}} P(\z_Q),\]
where $\calQ_{\D}$ is a set of representatives of positive definite binary integral quadratic forms $Q(x,y)=ax^2+bxy+cy^2$ of discriminant ${\D}:=-24n+1$ satisfying $6\vert a$, modulo the action of the group $\Gamma_0(6)$, where we choose a set of representatives such that $b\equiv 1\smod{12}$. Moreover, $\z_Q$ is the unique point in the upper half-plane $\bbH$ satisfying $Q(\z_Q,1)=0$, and $P$ is a certain Maass form of weight $0$ defined in \eqref{P}.

In particular, Bruinier and Ono proved that $P(\z_Q)$ is an algebraic number, and in fact it is known that $(24n-1)P(\z_Q)$ is an algebraic integer \cite{LR}. For related work studying and applying formulas for the partition function in terms of traces of singular moduli, the interested reader is also referred to \cite{BringmannOno,FolsomMasri,Lee}. In this context, it is natural to define the polynomial
\begin{gather}\label{Hn}
H_{{\D}}(x):=\prod\limits_{Q\in\calQ_{{\D}}}(x-P(\z_Q))\in\Q[x].
\end{gather}
By an elementary calculation (see Lemma 3.7 of \cite{BOS}), this factors over $\Q$ as
\[H_{{\D}}(x)=\prod\limits_{\substack{f>0 \\ f^2|{{\D}}}} \eps(f)^{h\left(\frac {{\D}}{f^2}\right)} \widehat H_{\frac{{{\D}}}{f^2}}(\eps(f)x),\]
where $h(d)$ denotes the class number of discriminant $d$, $\eps(f)=1$ if $f\equiv \pm 1\smod{12}$ and $\eps(f)=-1$ otherwise. We also set 
\begin{gather}\label{Hn'}
\widehat{H}_{{\D}}(x):=\prod\limits_{Q\in \P_{\D}} (x-P(\z_Q)),
\end{gather}
where $\P_{\D}$ is the set of \emph{primitive} forms in $\calQ_{{\D}}$, i.e., those forms for which $\gcd (a,b,c) =1$.

Bruinier and Ono \cite{BO} asked whether $\widehat H_{\D}(x)$ is irreducible and in \cite{BOS} they produced, together with Sutherland, very strong numerical evidence for the affirmative answer. In this paper, we settle their question.
\begin{theorem}\label{main}
The polynomial $\widehat H_{{\D}}(x)$ is irreducible over $\Q$. Moreover we have that 
$\Omega_{t}$ is the splitting field of $\widehat H_{{\D}}(x)$,
where ${\D}=-24n+1=t^2 d$ with $d$ a fundamental discriminant and $\Omega_t$ the ring class field of the order of conductor $t$ in $K:=\Q(\sqrt{d})$.
\end{theorem}

\begin{remark*}
It should be possible to give a general version of Theorem \ref{main} for almost holomorphic modular functions with rational expansions at the cusps, however, the authors have chosen to highlight this special case for the sake of explicitness, and in particular in order to establish the irreducibility for \emph{all} polynomials, as opposed to all but finitely many in the general case.
\end{remark*}

The paper is organized as follows.  
In \Cref{Prelim} we recall some tools required for the proof of \Cref{main} such as Masser's formula and a convenient form of Shimura reciprocity due to Schertz. The proof itself is the subject of \Cref{Irr}. 
\section*{Acknowledgements}

\noindent The authors are grateful to Kathrin Bringmann  and Andrew Sutherland for many useful conversations and comments which greatly improved the exposition of this paper. They would also like to thank the anonymous referee for helpful comments.
\section{Preliminaries}\label{Prelim}

\subsection{Masser's formula}\label{Masser}
Throughout, $\tau=u+iv\in\bbH$ (the complex upper half-plane) with $u,v\in\R$, and let $q:=e^{2\pi i\tau}$. For a function $f:\bbH\rightarrow\C$, an integer $k$, and a matrix $\gamma=\abcd\in\GL_2(\R)$ with $\det\gamma>0$ we define the weight $k$ \emph{slash operator} by
\[f|_k\gamma(\tau):=\det(\gamma)^{\frac k2}(c\tau+d)^{-k}f\left(\frac{a\tau+b}{c+tau+d}\right).\]  
Denote by $E_k$ the normalized weight $k$ Eisenstein series for $\SLZ$ with leading coefficient $1$, and let 
\[j(\tau):=1728\frac{E_4^3(\tau)}{E_4^3(\tau)-E_6^2(\tau)}\]
be the classical $j$-invariant. Furthermore, 
\[\eta(\tau):=q^\frac{1}{24}\prod\limits_{\ell=1}^\infty \left(1-q^\ell\right)\]
is the Dedekind eta function. We require the function
\begin{gather}\label{F}
F(\tau):=\frac{E_2(\tau)-2E_2(2\tau)-3E_2(3\tau)+6E_2(6\tau)}{2\eta(\tau)^2\eta(2\tau)^2\eta(3\tau)^2\eta(6\tau)^2}=q^{-1}-10-29q-104q^2-...,
\end{gather}
which is a weakly holomorphic modular form of weight $-2$ for $\Gamma_0(6)$. The Maass raising operator 
\[R_k:=\frac{1}{2\pi i}\frac{\partial}{\partial\tau}-\frac{k}{4\pi v},\]
maps this weakly holomorphic modular form $F$ to a non-holomorphic modular function
\begin{gather}\label{P}
P(\tau):=-R_{-2}(F)(\tau),
\end{gather}
which is an eigenfunction of the hyperbolic Laplacian
\[\Delta:=-v^2\left(\frac{\partial^2}{\partial u^2}+\frac{\partial^2}{\partial v^2}\right)\]
with eigenvalue $-2$.

As in \cite{LR}, we decompose $P$ as
\[P=A+B\cdot C\]
where $A$ and $B$ are modular functions for $\Gamma_0(6)$ given by
\begin{align}
\label{A} A(\tau)&:=-\frac{1}{2\pi i}\cdot\frac{\partial}{\partial\tau} F(\tau)-\frac 16 E_2(\tau)F(\tau)+\frac{F(\tau)E_6(\tau)(7j(\tau)-6912)}{6E_4(\tau)(j(\tau)-1728)},\\
\label{B} B(\tau)&:=\frac{F(\tau)E_6(\tau)j(\tau)}{E_4(\tau)},
\end{align}
and $C$ is a  non-holomorphic modular function for $\SLZ$, given by
\begin{gather}\label{C}
C(\tau):=\frac{E_4(\tau)}{6E_6(\tau)j(\tau)}\left(E_2(\tau)-\frac{3}{\pi v}\right)-\frac{7j(\tau)-6912}{6j(\tau)(j(\tau)-1728)}.
\end{gather}
This decomposition is especially useful as Masser gives a very important formula for the singular moduli of $C$ in Appendix I in \cite{Masser}. To state this result, we recall the \emph{modular polynomial} $\Phi_{-D}(x,y)$ which is defined for a discriminant $D<0$ by the relation
\[\Phi_{-D}(j(\tau),y):=\prod\limits_{M\in \mathcal{V}} (y-j(M\tau)).\]
Here, $\mathcal{V}$ is a system of representatives of $\SLZ\setminus \Gamma_{-D}$, where $\Gamma_{-D}$ denotes the set of all primitive integral $2\times 2$-matrices of determinant $-D$ (see e.g. \cite{KK}, Section IV.1.6). It is a well-known fact (cf. \cite{Lang}, Chapter 5, Section 2) that $\Phi_{-D}(x,y)\in\Z[x,y]$. Now let $Q$ be a quadratic form\footnote{From here on, the term ``quadratic form'' always means ``positive definite integral binary quadratic form'' if not declared otherwise.} of discriminant $D$ and $\z_Q$ the corresponding $CM$-point. Then we define numbers $\beta_{\mu,\nu}(\z_Q)$ via the power series expansion 
\[\Phi_{-D}(x,y)=:\sum\limits_{\mu,\nu} \beta_{\mu,\nu}(\z_Q) (x-j(\z_Q))^\mu(y-j(\z_Q))^\nu.\]
It is easy to see that all these numbers lie in the field $\Q(j(\z_Q))$ and that $\beta_{\mu,\nu}(\z_Q)=\beta_{\nu,\mu}(\z_Q)$. Masser's formula then states that in the case that the discriminant $D$ is not \emph{special}, i.e., not of the form $D=-3d^2$, one has
\begin{gather}\label{eqMasser}
C(\z_Q)=\frac{2\beta_{0,2}(\z_Q)-\beta_{1,1}(\z_Q)}{\beta_{1,0}(\z_Q)}.
\end{gather}
Note in particular that for any $n$, $\D=-24n+1$ is not special. Using the formula of Masser, we are able to reduce our problem of studying singular moduli for nonholomorphic modular functions to the study of holomorphic modular functions associated to each discriminant, as in the following result.  

\begin{lemma}\label{MD}
For every non-special discriminant $D<0$, there exists a (meromorphic) modular function $M_D$ for $\Gamma_0(6)$ such that
\[P(\z_Q)=M_D(\z_Q)\]
for all quadratic forms $Q$ of discriminant $D$.
\end{lemma}
\begin{proof}
Using the definition of the modular polynomial one finds explicitly that (see also (2.9) in \cite{BOS})
\begin{align}
\label{beta01}\beta_{0,1}(\z_Q)&=[y]\Phi_{-D}(j(\z_Q),y+j(\z_Q)), \\
\label{beta11}\beta_{1,1}(\z_Q)&=[y]\Phi'_{-D}(j(\z_Q),y+j(\z_Q)), \\
\label{beta02}\beta_{0,2}(\z_Q)&=[y^2]\Phi_{-D}(j(\z_Q),y+j(\z_Q)),
\end{align}
where $\Phi_{-D}':=\tfrac{\partial}{\partial x}\Phi_{-D}$ and $[y^k]\Phi_{-D}(x,y)$ denotes the coefficient of $y^k$ in $\Phi_{-D}$. The right-hand sides of \eqref{beta01}--\eqref{beta02} clearly also make sense if we replace the algebraic number $j(\z_Q)$ by the modular function $j(\tau)$, yielding modular functions $\beta_{\mu,\nu}(\tau)$ for $\SLZ$. Thus the function
\begin{gather}\label{eqMD}
M_D(\tau):=A(\tau)+B(\tau) \frac{2\beta_{0,2}(\tau)-\beta_{1,1}(\tau)}{\beta_{1,0}(\tau)}
\end{gather}
has the desired properties.
\end{proof}

\subsection{$N$-systems}\label{Schertz}
In this subsection, we cite some results of Schertz \cite{Schertz}. 
We begin by recalling a convenient set of representatives for primitive quadratic forms introduced in \cite{Schertz} for the study of class invariants known as Weber's class invariants. 
\begin{definition}
Let $N\in\N$ and $D=t^2d<0$ be a discriminant, with $t\in\N$ and $d$ a fundamental discriminant. Moreover, let $\{Q_1,...,Q_r\}$ ($Q_j(x,y)=a_jx^2+b_jxy+c_jy^2$) be a system of representatives of primitive quadratic forms modulo $\SLZ$. We call the set $\{Q_1,...,Q_r\}$ an \emph{$N$-system mod $t$} if the conditions
\[\gcd(c_j,N)=1\quad\text{and}\quad b_j\equiv b_\ell\smod{N},\: 1\leq j,\ell\leq r\]
are satisfied.
\end{definition}
\begin{remark*}
Schertz gave this definition (see \cite{Schertz}, p. 329) in terms of ideal classes of the ring class field and with switched roles of $a$ and $c$. He also proved constructively that an $N$-system mod $t$ always exists (see \cite{Schertz}, Proposition 3).
\end{remark*}
The following theorem (see \cite{Schertz}, Theorem 4) is a key in the proof of \Cref{main}. 
\begin{theorem}[Schertz]\label{theoSchertz}
Let $g$ be a modular function for $\Gamma_0(N)$ for some $N\in\N$ whose Fourier coefficients at all cusps lie in the $N$th cyclotomic field. Suppose furthermore that $g(\tau)$ and $g(-\tfrac{1}{\tau})$ have rational Fourier coefficients, and let $Q(x,y)=ax^2+bxy+cy^2$ be a quadratic form of discriminant $D=t^2d$, $d$ a fundamental discriminant, with $\gcd(c,N)=1$ and $N\vert a$. Then we have that $ g(\tau_Q)\in\Omega_t$ unless $g$ has a pole at $\tau_Q$.

Moreover, if $\{Q=Q_1,...,Q_h\}$ is an $N$-system mod $t$, then 
\[\{g(\tau_{Q_1}),...,g(\tau_{Q_h}\}=\{\sigma(g(\tau_{Q_1}))\: :\: \sigma\in\Gal_D\}\]
where $\Gal_D$ denotes the Galois group of $\Omega_t/\Q(\sqrt{d})$.
\end{theorem}
\begin{remark*}
Schertz stated the above theorem for modular functions for the group \[\Gamma^0(N):=\left\{\left(\begin{matrix}a & b\\ c&d\end{matrix}\right)\in\SLZ\: : \: b\equiv 0\smod{N}\right\}\]
which is clearly isomorphic to $\Gamma_0(N)$ via conjugation with $S:=\left(\begin{smallmatrix} 0 & -1 \\ 1 & 0 \end{smallmatrix}\right)$. 
This conjugation must be carried over to the quadratic forms as well which explains the change of roles of the coefficients $a$ and $c$ compared to \cite{Schertz}.
\end{remark*}

\subsection{Poincar\'e series}\label{secPoincare}
In this subsection we briefly recall some important facts about Maass-Poincar\'e series. For more details, we refer to the survey in Section 8.3 of \cite{Ono} and the earlier works \cite{Fay, Hejhal}. 

Let $M_{\nu,\mu}$ denote the usual $M$-Whittaker function (see e.g. \cite{GR}, p. 1014) and define for $v>0$, $k\in\Z$, and $s\in\C$ the function
\[\calM_{s,k}(v):=v^{-\frac{k}{2}}M_{-\frac k2,s-\frac 12}(v).\]
Using this, we construct the following Poincar\'e series for $\Gamma_0(N)$,
\begin{gather}\label{Poincare}
P_{m,s,k,N}(\tau):=\frac{1}{2\Gamma(2s)}\sum\limits_{\gamma\in\Gamma_\infty\setminus\Gamma_0(N)} [\calM_{s,k}(4\pi mv)e^{-2\pi imu}]|_k\gamma,
\end{gather}
where $m\in\N$, $\tau=u+iv\in\bbH$, $\Re(s)>1$, $k\in-\N$, and $\Gamma_\infty:=\left\{\pm\left(\begin{smallmatrix} 1 & n \\ 0 & 1 \end{smallmatrix}\right)\: : \: n\in\Z\right\}$. It is not hard to check (see e.g. Proposition 2.2 in \cite{BO}) that under the Maass raising operator, the Poincar\'e series $P_{m,s,k,N}$ is again mapped to a Poincar\'e series
\begin{gather}\label{PoincareRaising}
R_k\left(P_{m,s,k,N}\right)=m\left(s+\frac k2\right)P_{m,s,k+2,N}.
\end{gather}
In the special case when $k<0$ and $s=1-\tfrac k2$, the series $P_{m,k,N}:=P_{m,1-\frac k2,k,N}$ defines a harmonic Maass form of weight $k$ for $\Gamma_0(N)$ whose principal part at the cusp $\infty$ is given by $q^{-m}$, and at all the other cusps the principal part is $0$. In this situation, we have the following explicit Fourier expansion at $\infty$ (see e.g. \cite{Ono}, Theorem 8.4).
\begin{proposition}\label{PoincareFourier}
For $m,N\in\N$, $k\in-\N$, and $\tau\in\bbH$ we have
\[(1-k)!P_{m,k,N}\left(\tau\right)=(k-1)(\Gamma(1-k,4\pi mv)-\Gamma(1-k))q^{-m}+\sum\limits_{\ell\in\Z} b_{m,k,N}(\ell,v)q^\ell,\]

where the coefficients 
$b_{m,k,N}\left(\ell,v\right)$ are defined as follows. 

\noindent 1) If $\ell<0$, then
\begin{displaymath}
\begin{split}
b_{m,k,N}(\ell,v)
  =2 \pi i^{2-k}  (k-1) \, &\Gamma(1-k,4 \pi |\ell| v)
  \left|\frac{\ell}{m}\right|^{\frac{k-1}{2}}\\
 &\ \ \ \ \times  \sum_{\substack{c>0\\c\equiv 0\smod{N}}}
   \frac{K(-m,\ell,c)}{c}\cdot
 J_{1-k}\!\left(\frac{4\pi\sqrt{|m\ell|}}{c}\right),
\end{split}
\end{displaymath}
where 
\[\Gamma(\alpha;,x):=\int_x^{\infty}e^{-t}t^{\alpha-1}dt.\]

\noindent
2) If $\ell>0$, then
$$
b_{m,k,N}(\ell,v)= - 2 \pi  i^{2-k} (1-k)!\ell^{\frac{k-1}2}m^{\frac{1-k}2}
  \sum_{\substack{c>0\\c\equiv 0\smod{N}}}
   \frac{K(-m,\ell,c)}{c}\cdot
 I_{1-k}\!\left(\frac{4\pi\sqrt{|m\ell|}}{c}\right).
$$

\noindent 3) If $\ell=0$, then
$$
b_{m,k,N}(0,v)=-(2\pi i)^{2-k} m^{1-k} \sum_{\substack{c>0\\c\equiv
0\smod{N}}}
 \frac{K(-m,0,c)}{c^{2-k}}.
$$
Here, $I_s$ and $J_s$ denote the usual $I$- and $J$-Bessel functions and $K(m,\ell,c)$ is the usual Kloosterman sum,
\[K(m,\ell,c):=\sum\limits_{d\smod c^*} \exp\left(2\pi i\left(\frac{m\overline{d}+\ell d}{c}\right)\right),\]
where $d$ runs through the residue classes $\smod c$ which are coprime to $c$ and $\overline{d}$ denotes the multiplicative inverse of $d\smod c$.
\end{proposition}

\section{Irreducibility of $\widehat H_{\D}(x)$}\label{Irr}
Now that we have recalled the relevant known facts, we proceed towards proving our main result, \Cref{main}. We require some information about the Fourier expansions of $A$, $B$, and $M_D$, from equations \eqref{A}, \eqref{B}, and \eqref{MD} respectively, at all cusps, which we provide in the following lemma. 
\begin{lemma}\label{cusps}
\begin{enumerate}
\item The modular functions $A$ and $B$ have Fourier expansions with rational Fourier coefficients at all cusps of $\Gamma_0(6)$.
\item If $D<0$ is a non-special discriminant, then the Fourier coefficients of $M_D$ at all cusps lie in the field $\Q(\zeta_6)$, where $\zeta_6:=e^{\frac{2\pi i}{6}}$, and the Fourier coefficients of $M_D(\tau)$ and $M_D(-\tfrac 1\tau)$ are rational.
\end{enumerate}
\end{lemma}
\begin{proof}
($1$) The denominator of the function $F$ lies in the one-dimensional space $S_4(\Gamma_0(6))$, and hence is an eigenfunction of all Atkin-Lehner operators, the eigenvalue being always $1$. The numerator, which we denote by $F_1$, is a weight two modular form for $\Gamma_0(6)$ and is an eigenfunction of all Atkin-Lehner involutions as well. This can be seen by a short and direct calculation: 
For the convenience of the reader, we give some details. Namely, note that the Atkin-Lehner operators obtained by slashing with one of the following matrices
\[W_6:=\begin{pmatrix} 0 & -1 \\ 6 & 0 \end{pmatrix},\quad W_3:=\begin{pmatrix} 3 & 1 \\ 6 & 3 \end{pmatrix},\quad W_2:=\begin{pmatrix} 2 & -1 \\ 6 & -2 \end{pmatrix},\]
map the cusp $\infty$ to $0$, $\tfrac 12$, $\tfrac 13$, (respectively).
One finds by a direct calculation of the Smith normal form of the matrices $W_d$, $d=2,3,6$, that 
\begin{align*}
F_1|_2W_6=F_1, \quad F_1|_2W_3=-F_1, \quad\text{and }F_1|_2W_2=-F_1.
\end{align*}
Therefore, $F$ is an eigenfunction of all Atkin-Lehner involutions of level $6$.

Using this,we can directly calculate the Fourier expansion of $F$ at all cusps (correcting a typo in (3.2) of \cite{DM}). To this end we choose the following $12$ right coset representatives of $\SLZ/\Gamma_0(6)$ (where $T:=\left(\begin{smallmatrix} 1 & 1 \\ 0 & 1 \end{smallmatrix}\right)$),
\begin{gather}\label{gamma}
\begin{aligned}
\gamma_\infty&:=\begin{pmatrix} 1 & 0 \\ 0 & 1 \end{pmatrix},\\
\gamma_{\frac 13,r}&:=\begin{pmatrix} 1 & 0 \\ 3 & 1 \end{pmatrix}  T^r\qquad\text{for }r=0,1, \\
\gamma_{\frac 12,s}&:=\begin{pmatrix} 1 & 1 \\ 2 & 3 \end{pmatrix}  T^s\qquad\text{for }s=0,1,2, \\
\gamma_{0,t}&:=\begin{pmatrix} 0 & -1 \\ 1 & 0 \end{pmatrix} T^t\qquad\text{for }t=0,1,2,3,4,5.
\end{aligned}
\end{gather}
Using the relation $W_6=\gamma_{0,0}V_6$, where $V_d:=\left(\begin{smallmatrix} d & 0 \\ 0 & 1\end{smallmatrix}\right)$, we find that
\[F(\tau)=F|_2W_6(\tau)=\frac{1}{6}F|_2\gamma_{0,0} (6\tau).\]
Thus we find the following Fourier expansions of $F$ at the cusp $0$,
\begin{gather}\label{cusp0}
F|_2\gamma_{0,t}(\tau)=6\sum\limits_{m=-1}^\infty a_m\zeta_6^{tm} q^\frac{m}{6},
\end{gather}
where $F(\tau)=:\sum\limits_{m=-1}^\infty a_mq^m.$
Similarly, we have
\[W_2= \gamma_{0,0}V_2\gamma_{\frac 12,0}^{-1},\qquad W_3=\gamma_{0,0}^{-1} V_3T^{-1}\gamma_{\frac 13,0}^{-1}\begin{pmatrix} 1 & 0 \\ 6 & 1 \end{pmatrix},\]
which yields
\[F(\tau)=-F|_2W_2(\tau)=-6F|_2V_2\gamma_{\frac 12,0}^{-1}\left(\frac \tau 6\right)=-3F|_2\gamma_{\frac{1}{2},0}\left(\frac \tau 3\right).\]
Hence, at $\frac12$, we have the Fourier expansions
\begin{gather}\label{cusp12}
F|_2\gamma_{\frac 12,s} (\tau)=3\sum\limits_{m=-1}^\infty a_m \zeta_6^{3+2ms}q^\frac m3.
\end{gather}

\noindent
A similar calculation yields the Fourier expansions at $\tfrac 13$:
\begin{gather}\label{cusp13}
F|_2\gamma_{\frac 13,r}(\tau)=2\sum\limits_{m=-1}^\infty a_m (-1)^{mr} q^\frac m2.
\end{gather}
Thus, it is clear that $B$ has a Fourier expansion with rational coefficients at all cusps, since $\tfrac{B}{F}$ is a weight $2$ meromorphic modular form for $\SLZ$ which has rational (in fact integral) Fourier coefficients at $\infty$. 

The Serre derivative of $F$, which is given by $\frac{1}{2\pi i}\frac{\partial}{\partial\tau} F(\tau)+\frac 16 E_2(\tau)F(\tau)=:A_1(\tau)$ has a rational leading term at $\infty$, namely $\tfrac 76$, and integer coefficients otherwise (note that for $n>0$ the $n$-th coefficient of $E_2$ is divisible by $24$). The same is true for the function $$A_2(\tau):=\frac{F(\tau)E_6(\tau)(7j(\tau)-6912)}{6E_4(\tau)(j(\tau)-1728)}$$ (again, note that all but the zeroth Fourier coefficient of $E_6$ are divisible by $504=6\cdot 84$). Therefore, $A=-A_1+A_2$, has integer coefficients at $\infty$. Since the Serre derivative commutes with the action of $\SLZ$ and $A_2$ is the product of $F$ and a level $1$ form, the same argument as for $B$ yields the claim.

\noindent($2$) It remains to show that the level $1$ modular function
\[\frac{2\beta_{0,2}(\tau)-\beta_{1,1}(\tau)}{\beta_{1,0}(\tau)}\]
from \eqref{eqMD} has rational Fourier coefficients at $\infty$. But since by definition the functions $\beta_{\mu,\nu}(\tau)$ are polynomials in $j(\tau)$ with rational coefficients, this is clear.
\end{proof}

With this, we can immediately show the following.

\begin{proposition}\label{propGalois}
Let $P$ be as in as in the introduction, ${{\D}}=-24n+1=t^2d$ with $d$ a fundamental discriminant. Then $P(\tau_{Q_0})\in\Omega_t$ and
\[\{P(\tau_Q) \: :\: Q\in \P_{\D}\}=\{\sigma(P(\tau_{Q_0})) \: : \: \sigma\in\Gal_{{\D}}\}\]
for all $Q_0\in \P_{{\D}}$. In particular, the values $P(\z_Q)$ for $Q\in \P_{{\D}}$ generate the ring class field $\Omega_t$ over $\Q(\sqrt{d})$ and $\widehat H_{{\D}}(x)$ is a perfect power of an irreducible polynomial.
\end{proposition}
\begin{proof}
By the proposition on p. 505 of \cite{GKZ}, we know that the set $\P_{{\D}}$ also represents the $\SLZ$-equivalence classes of primitive quadratic forms of discriminant $\D$. Hence $\P_{{\D}}$ is a $6$-system mod $t$, and by \Cref{MD} we know that for all $Q\in \P_{{\D}}$ we have $P(\z_Q)=M_{{\D}}(\z_Q)$ for some modular function $M_{{\D}}$ for $\Gamma_0(6)$. By \Cref{cusps}, this function satisfies the conditions in \Cref{theoSchertz}, which immediately gives the first part of the proposition. The second part follows from elementary Galois theory and the fact that any finite extension of $\Q(\sqrt{d})$ is separable.
\end{proof}

The remainder of this section is devoted to the proof that $\widehat H_{{\D}}$ is itself an irreducible polynomial.  By \Cref{propGalois}, this is equivalent to showing that there is at least one singular modulus $P(\z_Q)$ which is assumed only once as $Q$ ranges over $\P_{{\D}}$. To see this, we use an effective version of Lemma 5 in \cite{DM}. For this, let
\begin{align*}
\F:=&\left\{\tau\in \bbH \: : \: -\frac{1}{2}< \Re(\tau)\leq \frac 12\text{ and }|\tau|>1\right\}\\
&\qquad\qquad\qquad\qquad\cup\left\{\tau\in\bbH\: :\: |\tau|=1\text{ and }0\leq \Re(\tau)\leq \frac 12\right\}
\end{align*}
denote the usual fundamental domain for the action of $\SLZ$ on $\bbH$. Our main analytic result is then the following, which is an effective form of a result of Dewar and Murty (see Lemma 5 in \cite{DM}).
\begin{lemma}\label{bound}
Let $\gamma$ be one of the $12$ matrices in \eqref{gamma} and assume that
\[F|_{-2}\gamma(\tau)=h\zeta q^{-\frac 1h} +a_0+a_1q^\frac 1h+...,\]
where $h$ is the width of the cusp $\gamma\infty$ and $\zeta$ is a certain $6$th root of unity. Then, for every $\gamma$ and $\tau\in\F$, we have
\[P|_0\gamma(\tau)=\zeta\left(1-\frac{h}{2\pi v}\right)e^{-\frac{2\pi i\tau}{h}}+E_\gamma(\tau),\]
where the error $E_\gamma$ satisfies the uniform bound $|E_\gamma(\tau)|\leq \kappa:= 1334.42$. 
\end{lemma}
\begin{proof}
We can express the function $F$ in terms of Poincar\'e series via
\[F=P_{1,-2,6}-P_{1,-2,6}|_{-2}W_2-P_{1,-2,6}|_{-2}W_3+P_{1,-2,6}|_{-2}W_6,\]
(see \cite{BO}, p. 213). Then note that $|P_{1,-2,6}|_{-2}W_d|$, $d=1,2,3,6$ is majorized by $$\left|q^{-1}\right|+\sum_{\ell\in\Z} \left|b_{1,-2,6}(\ell,v)\right|\cdot |q|^\ell$$ with $b_{1,-2,6}(\ell,v)$ as given in \Cref{PoincareFourier}. Indeed, the Maass-Poincar\'e series, defined in \Cref{secPoincare}, grow exponentially approaching the cusp $\infty$ and have moderate growth at the other cusps. Following \cite{Rankin} (see in particular Chapter 5 there), the first author and Ono \cite{BrOno} introduced Maass-Poincar\'e series which grow exponentially at any given cusp and grow moderately at all the others. In \cite{Rankin}, Theorem 5.1.2., the function $\widetilde{P}|_0\sigma$ with $\widetilde{P}$ a Poincar\'e series for a group $\Gamma$ and $\sigma\in\GL_2(\Q)$ with $\det(\sigma)>0$ is expressed in terms of a Poincar\'e series for the group $\sigma^{-1}\Gamma\sigma$, possibly growing at a different cusp, see also \Cref{remRankin} below. The explicit Fourier expansions of these series, given in Theorem 3.2 of \cite{BrOno}, easily implies our claim because the Atkin-Lehner operators $W_d$ normalize $\Gamma_0(6)$. Futhermore, we can ignore the non-holomorphic parts of the Maass-Poincar\'e series because $F$ is holomorphic on $\bbH$, and thus they must cancel.

Thus, we are left to bound the Fourier coefficients of the Poincar\'e series $P_{1,-2,6}(\tau)=q^{-1}+\tfrac{1}{6}\sum_{\ell\in\Z} b_{1,-2,6}(\ell,v)q^\ell$ (for $\ell\geq 0$) from above to obtain an upper bound for our error function $E_\gamma$.  Note that our estimates could be improved, but for the purpose of this paper, we do not require this. Furthermore, we make all constants explicit, which is crucial for our purposes in order to obtain the absolute bound $\kappa$.
Clearly, we may estimate
\begin{equation}\label{KloosterTrivBound}|K(m,\ell,c)|\leq c.\end{equation}
Furthermore, for $x>0$ and $\nu>-\tfrac 12$ we have the bound (eq. (6.25) in \cite{Luke})
\begin{gather}\label{Besselbound}
I_\nu(x)<\frac{1}{\Gamma(\nu+1)}\left(\frac{x}{2}\right)^\nu\cosh(x).
\end{gather}
which, for $0<x\leq 1$, implies
\begin{equation}\label{BesselBound2}I_\nu(x)\leq \frac{2}{\Gamma(\nu+1)}\left(\frac x2\right)^\nu. \end{equation}
For $x\geq 1$, we can bound
\begin{equation}\label{BesselBound3}I_\nu(x)\leq \frac{e^x}{\sqrt{2\pi x}},\end{equation}
which follows from the well-known asymptotic formula for $I_\nu$ (see e.g. \cite{GR}, eq. 8.451.5). We now split the expression for $b_{1,-2,6}(\ell,v)$ in \Cref{PoincareFourier} and use the above bounds \eqref{KloosterTrivBound},\eqref{BesselBound2},\eqref{BesselBound3} to obtain for $\ell\in\N$ that
\begin{align*}
|b_{1,-2,6}(\ell,v)| &\leq 12\pi \ell^{-\frac 32} \left[ \frac{2\pi \sqrt{\ell}}{3} \frac{\exp\left(\frac{2\pi\sqrt{\ell}}{3}\right)}{2\pi\sqrt{\frac{\sqrt{\ell}}{3}}}+\sum\limits_{c>\frac{2\pi \sqrt{\ell}}{3}} \frac{1}{3}\left(\frac{\pi\sqrt{\ell}}{3c}\right)^3\right]\\
          &\leq 12\pi \ell^{-\frac 32} \left[\frac{\ell^\frac 14}{\sqrt{3}}  \exp\left(\frac{2\pi\sqrt{\ell}}{3}\right) + \frac{\pi^3}{81}\ell^\frac{3}{2}\zeta(3)\right].
\end{align*}
Furthermore, it follows directly that
\[|b_{1,-2,6}(0,v)|\leq \frac{2}{27}\pi^4\zeta(3).\]
Now let $\tau\in\F$, which, in particular, implies that $v\geq \tfrac{\sqrt{3}}{2}$, and $\gamma$ one of the matrices in \eqref{gamma}. Then, as applying the matrix $\gamma$ to $F$ has the effect of multiplying its absolute value by the width $h$ of the cusp $\gamma\infty$, we find that 
{\allowdisplaybreaks
\begin{gather}\label{estimate}
\begin{aligned}
 &|E_\gamma(\tau)|=\left|P|_0 \gamma(\tau)-\zeta \left(1-\frac{h}{2\pi v}\right)q^{-\frac 1h}\right|= \left|R_{-2}\left(F|_{-2}\gamma (\tau)-\zeta hq^{-\frac 1h}\right)\right|\\
 \leq& \frac{4h}{6}\sum\limits_{\ell=1}^\infty \ell|b_{1,-2,6}(\ell,v)|\cdot |q|^{\frac \ell h}+\frac{2h}{6\pi v}\sum\limits_{\ell=1}^\infty |b_{2,-1,6}(\ell,v)|\cdot |q|^\frac \ell h\\
 \leq&   \frac{2h}{3}\sum\limits_{\ell=1}^\infty 12\pi \ell^{-\frac 12} \left[\frac{\ell^\frac 14}{\sqrt{3}}  \exp\left(\frac{2\pi\sqrt{\ell}}{3}\right) + \frac{\pi^3}{81}\ell^\frac{3}{2}\zeta(3)\right]e^{-\frac{\pi \sqrt{3} \ell}{h}} \\
& +\frac{2h}{3\pi \sqrt{3}} \left(\frac{2}{27}\pi^4\zeta(3) +\sum\limits_{\ell=1}^\infty 12\pi \ell^{-\frac 32} \left[\frac{\ell^\frac 14}{\sqrt{3}}  \exp\left(\frac{2\pi\sqrt{\ell}}{3}\right) + \frac{\pi^3}{81}\ell^\frac{3}{2}\zeta(3)\right]e^{-\frac{\pi \sqrt{3} \ell}{h}}\right)
\end{aligned}
\end{gather}
}
\noindent
We can clearly estimate this last expression from above by setting $h=6$, which we now do. It is easy to estimate the remaining infinite series in \eqref{estimate} by geometric series to obtain our asserted bound $\kappa$.
\end{proof}
\begin{remark}\label{remRankin}
It should be pointed out that in Theorem 5.1.2 in \cite{Rankin} cited in the proof above, a different type of Poincar\'e series is considered and also in the formulation  of the theorem, only actions of matrices in $\SLZ$ rather than Atkin-Lehner involutions on Poincar\'e series are studied. However, a close inspection of the proof reveals that the result is directly applicable to the situation of the Maass-Poincar\'e series $P_{1,-2,6}$ acted upon by the Atkin-Lehner involution $W_d$, $d=2,3,6$. 
\end{remark}
We are now in position to prove our main result, Theorem \ref{main}.
\begin{proof}[Proof of \Cref{main}]
Let $Q(x,y)=ax^2+bxy+cy^2=:[a,b,c]$ be a $\SLZ$-reduced quadratic form of discriminant ${{\D}}=-24n+1$, i.e., $b^2-4ac={{\D}}$, $|b|\leq a\leq c$, and $b>0$ whenever $|b|=a$ or $a=c$. Then exactly one matrix $\gamma_Q$ from those $12$ in \eqref{gamma} satisfies $Q\circ\gamma_Q^{-1}\in \P_{{\D}}$, see Lemma 3 and Table 1 in \cite{DM}. Further define the number $h_Q\in\{1,2,3,6\}$ and the $6$th root of unity $\zeta_Q$ by the relation
\[F|_{-2}\gamma_Q(\tau)=h_Q\zeta_Qq^{-\frac 1{h_Q}}+O(1).\]
Then one easily observes that $a\cdot h_Q\equiv 0\smod 6$ holds in all possible cases. We focus on the case $a\cdot h_Q=12$. Tables \ref{table1} and \ref{table2} list the quadratic forms $Q$ together with the associated data $\gamma_Q$, $h_Q$, and $\zeta_Q$, which are easily read off Table 1 in \cite{DM} and \eqref{cusp0}--\eqref{cusp13}. 

\begin{table}
\begin{tabular}{c|c|c|c|c}
$Q$ & $\gamma_Q$ & $h_Q$ & $\zeta_Q$ & $\varphi_Q$\\
\hline
$[2,-1,3n]$ & $\gamma_{0,0}$ & $6$ & $1$ & $-\tfrac{\pi}{12}$ \\
 &&&& \\
$[4,1,\tfrac{3n}{2}]$ & $\gamma_{\frac 12,1}$ & $3$ & $\zeta_6$ & $\frac{5\pi}{12}$ \\
 &&&& \\
$[6,-5,n+1]$ & $\gamma_{\frac 13,0}$ & $2$ & $1$ & $-\frac{5\pi}{12}$ \\
 &&&& \\
$[12,1,\tfrac n2]$ & $\gamma_\infty$ & $1$ & $1$ & $\tfrac{\pi}{12}$
\end{tabular}
\caption{Quadratic forms with $a\cdot h_Q=12$ for $n$ even}
\label{table1}
\end{table}

\begin{table}
\begin{tabular}{c|c|c|c|c}
$Q$ 
& $\gamma_Q$ & $h_Q$ & $\zeta_Q$ & $\varphi_Q$ \\
\hline
$[2,-1,3n]$ & $\gamma_{0,3}$ & $6$ & $-1$ & $\frac{11\pi}{12}$ \\
 &&&& \\
$[4,-3,\tfrac{3n+1}{2}]$ & $\gamma_{\frac 12,2}$ & $3$ & $\zeta_6^{-1}$ & $-\frac{7\pi}{12}$ \\
 &&&& \\
$[6,-5,n+1]$ & $\gamma_{\frac 13,1}$ & $2$ & $-1$ & $\frac{7\pi}{12}$ \\
\end{tabular}
\caption{Quadratic forms with $a\cdot h_Q=12$ for $n$ odd}
\label{table2}
\end{table}

Now for $\z_Q=-\tfrac b{2a}+\frac{\sqrt{24n-1}}{2a}i$ we find, by \Cref{bound}, that
\[|P|_0 \gamma_Q(\z_Q)|\in\left(M(n;a,h_Q)-\kappa, M(n;a,h_Q)+\kappa\right)\]
and
\[\arg(P_0\gamma_Q(\z_Q))\in\left(\varphi_Q-\arctan\left(\frac{\kappa}{M(n;a,h_Q)}\right),\varphi_Q+\arctan\left(\frac{\kappa}{M(n;a,h_Q)}\right)\right),\]
where 
\[M(n;a,h_Q)\left(1-\tfrac{2ah_Q}{2\pi\sqrt{24n-1}}\right) e^{2\pi \frac{\sqrt{24n-1}}{2ah_Q}}\]
and $\varphi_Q$ as given in Tables \ref{table1} and \ref{table2}  denotes the argument of the main term of $P|_0 \gamma_Q(\z_Q)$, which is given by $\arg(\zeta_Q)+\tfrac{b\pi}{12}$. If we now assume that $n\geq 54$, then we see that there cannot be any singular moduli with $a\cdot h_Q\neq 12$ matching those with $a\cdot h_Q=12$ because the difference 
\[\left|\left(1-\tfrac{2a_1h_{Q_1}}{2\pi\sqrt{24n-1}}\right)e^{2\pi \frac{\sqrt{24n-1}}{2a_1h_{Q_1}}}-\left(1-\tfrac{2a_2h_{Q_2}}{2\pi\sqrt{24n-1}}\right)e^{2\pi \frac{\sqrt{24n-1}}{2a_2h_{Q_2}}}\right|\]
would be larger than $2\kappa$ if $a_1h_{Q_1}=12$ and $a_2h_{Q_2}\neq 12$. The same is true for the distortion in the arguments among those which have $a\cdot h_Q=12$, since for $n\geq 54$ we have
\[\left|\arctan\left(\frac{\kappa}{M(n;a,h_Q)}\right)\right|\leq \frac{\pi}{24},\]
so in particular less than half of the smallest possible difference between two distinct $\varphi_Q$ from Tables \ref{table1} and \ref{table2}.
Hence all these values can only be simple zeros of $\widehat H_{{\D}}$ and by \Cref{propGalois}, the theorem is proven for $n\geq 54$. The remaining cases are checked by comparing with numerically computed examples.\footnote{In particular, we refer the reader to Sutherland's extensive table at \url{http://math.mit.edu/~drew/Pfiles/} for our test of the small cases.}. 
\end{proof}

\end{document}